\documentclass[11pt,twoside,letterpaper]{article} 
\usepackage{times,fancyhdr}   
\usepackage{amsfonts}                        
\usepackage{amsthm}                         
\usepackage{amsmath}                   
\usepackage{amssymb}  
\usepackage[normalem]{ulem}
\usepackage{color}
\usepackage{cases}

\numberwithin{equation}{section}


\makeatletter 
\def\cleardoublepage{\clearpage\if@twoside \ifodd\c@page\else%
    \hbox{}%
    \thispagestyle{empty}%
    \newpage%
    \if@twocolumn\hbox{}\newpage\fi\fi\fi} 
\makeatother 

\newtheorem{thm}{Theorem}[section] 

\newtheorem{cor}{Corollary}[section]
\newtheorem{lem}{Lemma}[section]

\newtheorem{rem}{Remark}[section]

\begin{document}
\title{
{\begin{flushleft}
\vskip 0.45in
{\normalsize\bfseries\textit{ }}
\end{flushleft}
\vskip 0.45in
\bfseries\scshape 
Linear independence of certain numbers in the base-$b$ number system
}}

\thispagestyle{fancy}
\fancyhead{}
\fancyhead[L]{In: Book Title \\ 
Editor: Editor Name, pp. {\thepage-\pageref{lastpage-01}}} 
\fancyhead[R]{ISBN 0000000000  \\
\copyright~2007 Nova Science Publishers, Inc.}
\fancyfoot{}
\renewcommand{\headrulewidth}{0pt}

\author{\bfseries\itshape 
Shintaro Murakami
\thanks{
Hirosaki University, Graduate School of Science and Technology, 
Hirosaki 036-8561, Japan \newline 
e-mail: h20ds203@hirosaki-u.ac.jp},\,\,\,
Yohei Tachiya
\thanks{Hirosaki University, Graduate School of Science and Technology, 
Hirosaki 036-8561, Japan 
\newline 
e-mail: tachiya@hirosaki-u.ac.jp}} 

\date{}
\maketitle

\begin{abstract}
Let $(i,j)\in \mathbb{N}\times \mathbb{N}_{\geq2}$ 
and 
$S_{i,j}$  
be an infinite subset of positive integers including all prime numbers 
in some arithmetic progression. 
In this paper, we prove the linear independence over $\mathbb{Q}$ of the numbers 
\[
1, \quad \sum_{n\in S_{i,j}}^{}\frac{a_{i,j}(n)}{b^{in^j}},\quad  (i,j)\in \mathbb{N}\times \mathbb{N}_{\geq2},
\]
where $b\geq2$ is an integer and  $a_{i,j}(n)$ are bounded nonzero integer-valued functions on $S_{i,j}$. 
Moreover, we also establish a necessary and sufficient condition on the subset 
$\mathcal{A}$ of $\mathbb{N}\times \mathbb{N}_{\geq2}$ 
for the numbers 
\[
1, \quad \sum_{n\in T_{i,j}}^{}\frac{a_{i,j}(n)}{b^{in^j}},\quad  (i,j)\in \mathcal{A}
\]
to be linearly independent over $\mathbb{Q}$ 
for any given infinite subsets $T_{i,j}$ of positive integers. 
Our theorems generalize a result of 
V.~Kumar. 
\end{abstract}
 


\pagestyle{fancy}  
\fancyhead{}
\fancyhead[EC]{Shintaro Murakami and Yohei Tachiya}
\fancyhead[EL,OR]{\thepage}
\fancyhead[OC]{Linear independence of certain numbers in the base-$b$ number system}
\fancyfoot{}
\renewcommand\headrulewidth{0.5pt}

\section{Introduction}\label{sec:2}

In 1996, Yu.~V.~Nesterenko~\cite{Nes} showed the lower bound for the transcendence degree of the field generated over $\mathbb{Q}$ by the values of the Eisenstein series. 
Nesterenko's theorem derives a number of remarkable 
transcendence and algebraic independence results for the 
values of various modular functions. 
For example, applying Nesterenko's theorem, D.~Bertrand~\cite{Ber} and independently D.~Duverney, Ke.~Nishioka, Ku.~Nishioka and I.~Shiokawa~\cite{DNNS} 
derived algebraic independence results for the values of the Jacobi theta functions. 
In particular, they proved that the number 
$\sum_{n=1}^{\infty}\alpha^{n^2}$ is transcendental for 
any algebraic number $\alpha$ 
with $0<|\alpha|<1$. 

On the other hand, we are not aware of the transcendence of the number 
$\sum_{n=1}^{\infty}\alpha^{n^k}$ for an integer $k\geq3$.  
Recently, V.~Kumar~\cite{Ku2} obtained the following linear independence result by using 
the properties of uniform distribution modulo $1$ of irrational numbers:

\begin{thm}[{\cite[Theorem~1]{Ku2}}]
\label{thm:0}
Let $k\geq2$, $b\geq2$ and $1\leq a_1<a_2<\cdots<a_m$ be integers such that 
$\sqrt[k]{a_i/a_j}\notin\mathbb{Q}$ for any $i\neq j$. Then the numbers 
\begin{equation}\label{eq:001}
1,\quad \sum_{n=1}^{\infty}\frac{1}{b^{a_1n^k}},\quad 
\sum_{n=1}^{\infty}\frac{1}{b^{a_2n^k}},\quad \dots, \quad 
\sum_{n=1}^{\infty}\frac{1}{b^{a_mn^k}}
\end{equation}
are linearly independent over $\mathbb{Q}$. 
\end{thm}

The aim of this paper is to extend Theorem~\ref{thm:0} 
by showing  
linear independence results for certain infinite series. 
In particular, 
as an application of our Theorem~\ref{thm:1}, 
we will generalize Theorem~\ref{thm:0} 
even without the condition $\sqrt[k]{a_i/a_j}\notin\mathbb{Q}$ for any $i\neq j$.   
This gives a positive answer to the question of V.~Kumar \cite[Section~4]{Ku2}. 
Moreover, we establish in Theorem~\ref{thm:3} a necessary and sufficient condition 
for a set of certain infinite series to be linearly independent over $\mathbb{Q}$. 
We will find by Theorem~\ref{thm:3} that 
the same conclusion of Theorem~\ref{thm:0} can be achieved in case where $k\geq3$ even if 
the infinite series in \eqref{eq:001} are replaced by any subseries (Corollary~\ref{cor:34}).

Before stating our Theorem~\ref{thm:1}, we prepare some notations. 
Let $\mathbb{N}$ denote the set of positive integers. 
Let $(i,j)\in \mathbb{N}\times \mathbb{N}_{\geq 2}$
and 
$S_{i,j}$ be an 
infinite subset of $\mathbb{N}$ including all prime numbers 
in the arithmetic progression $n\equiv h_{i,j}$ (mod $d_{i,j}$), 
where  $d_{i,j}$ and $h_{i,j}$ are positive integers and relatively prime.  
Note that Dirichlet's theorem  
states that such arithmetic progression contains infinitely many prime numbers. 
Thus, we can choose various infinite sets~$S_{i,j}$ such as 
the sets of all prime numbers 
in arithmetic progressions, all prime numbers, all squarefree integers and all positive integers. 

Let $S_{i,j}$ be the above infinite set and 
$a_{i,j}(n)$ be a nonzero integer-valued function on $S_{i,j}$.  
Define the function
\begin{equation}\label{eq:863}
f_{i,j}(z):=\sum_{n\in S_{i,j}}^{}a_{i,j}(n)z^{in^j},
\end{equation}
which converges in $|z|<1$. 
Since $j\geq2$, the function~\eqref{eq:863} is not rational. 
We establish linear independence results for the values of the functions~\eqref{eq:863} at the rational argument $z=1/b$ with an integer $b\geq2$. 
Recall that an infinite set of numbers is called linearly independent if each one of its finite subsets is linearly 
independent; otherwise it is called linearly dependent. 
Our results are the following. 
\begin{thm}\label{thm:1} 
Let $b\geq2$ be an integer and 
$f_{i,j}(z)$ be the functions defined in \eqref{eq:863}. 
Then the set of  the numbers 
\begin{equation}\label{eq:39746}
1, \qquad f_{i,j}(1/b)=\sum_{n\in S_{i,j}}^{}\frac{a_{i,j}(n)}{b^{in^j}},\quad (i,j)\in \mathbb{N}\times \mathbb{N}_{\geq 2}
\end{equation}
is linearly independent over $\mathbb{Q}$. 
\end{thm}

Applying Theorem~\ref{thm:1} with $S_{i,j}:=\mathbb{N}$ and $a_{i,j}(n):=(\pm 1)^n$ ($n\geq1$) for all $(i,j)\in \mathbb{N}\times \mathbb{N}_{\geq 2}$, we have  

\begin{cor}\label{cor:2}
Let $b\geq2$ be an integer. Then the set of the numbers 
\[
1, \qquad \alpha_{i,j}:=\sum_{n=1}^{\infty}\frac{1}{b^{in^j}}\quad 
(i=1,2,\dots,
\,\,\,
j=2,3,\dots)
\]
is linearly independent over $\mathbb{Q}$. 
The same holds for the set of the numbers 
\[
1, \qquad \beta_{i,j}:=\sum_{n=1}^{\infty}\frac{(-1)^n}{b^{in^j}}\quad 
(i=1,2,\dots,
\,\,\,
j=2,3,\dots).
\]
\end{cor}

Note that the first assertion in Corollary~\ref{cor:2} gives an unconditional version of Theorem~\ref{thm:0}. 

\begin{rem}\label{rem:0} 
Let $\alpha_{i,2}$ and $\beta_{i,2}$ be as in Corollary~\ref{cor:2}. 
Then we have the equalities $\alpha_{i,2}=(\vartheta_3(1/b^i)-1)/2$ and $\beta_{i,2}=(\vartheta_4(1/b^i)-1)/2$, where 
\[
\vartheta_3(q):=1+2\sum_{n=1}^{\infty}q^{n^2},\qquad 
\vartheta_4(q):=1+2\sum_{n=1}^{\infty}(-1)^nq^{n^2}
\]
are the Jacobi theta functions defined in $|q|<1$.  
Hence, Corollary~\ref{cor:2} provides linear independence results for the values of the Jacobi theta functions
\begin{equation}\label{theta}
\vartheta_\ell(1/b^i),\quad i=1,2,\dots
\end{equation}
for each $\ell\in\{3,4\}$. 
As mentioned at the beginning of this section, the numbers in \eqref{theta} are all transcendental (cf. \cite{Ber,DNNS}), and hence, so are the numbers $\alpha_{i,2}$, 
$\beta_{i,2}$ $(i=1,2,\dots)$.  
More strongly, any two numbers in the set 
$\{
\alpha_{i,2},\beta_{i,2}
\mid i=1,2,\dots\}$ 
are algebraically independent over $\mathbb{Q}$, 
while any three are not (cf. \cite{ELT, EKT}). 
For example, the explicit algebraic relations 
among the first three numbers 
$\alpha_{i,2}$ $(i=1,2,3)$ and $\beta_{i,2}$ $(i=1,2,3)$, respectively,  can be derived 
from the results in \cite[\S~4.1]{EKT}. 
Recently, C.~Elsner and V.~Kumar \cite{EK} obtained linear independence results 
over the field of algebraic numbers for 
some three numbers in $\{\alpha_{i,2}\mid i\geq1\}$.  
\end{rem}

Moreover, applying Theorem~\ref{thm:1} with the sets $S_{i,j}$ of all prime numbers and all squarefree integers, we have 

\begin{cor}\label{cor:3}
Let $b\geq2$ be an integer. Then the set of the 
numbers 
\[
1, \qquad \gamma_{i,j}:=\sum_{p:{\rm prime}}\frac{1}{b^{ip^j}}\quad 
(i=1,2,\dots,
\,\,\,
j=2,3,\dots)
\]
is linearly independent over $\mathbb{Q}$. The same holds for the set of the numbers 
\[
1, \qquad \sum_{n:{\rm squarefree}}\frac{1}{b^{in^j}}\quad 
(i=1,2,\dots,
\,\,\,
j=2,3,\dots).
\]
\end{cor}

It should be noted that 
D.~H.~Bailey, J.~M.~Borwein, R.~E.~Crandall and C.~Pomerance~\cite{BBCP} 
proved 
that 
the numbers $\alpha_{i,j}$ ($j\geq3$) and $\gamma_{i,j}$ $(j\geq2)$ 
are either transcendental or 
algebraic numbers with degrees at least $j$ and $j+1$, respectively. \\

In Theorem~\ref{thm:1} we can not replace the sets $S_{i,j}$ 
by  any infinite subsets 
of positive integers; indeed, the numbers \eqref{eq:39746} can be \textit{linearly dependent} over $\mathbb{Q}$ 
for suitable infinite sets $S_{i,j}$, e.g. clearly  
$\sum_{n\in 2\mathbb{N}}^{}{b^{-n^2}}=
\sum_{n\in \mathbb{N}}^{}{b^{-4n^2}}
$. 
The following Theorem~\ref{thm:3} gives a necessary and sufficient condition on 
the subset $\mathcal{A}$ of $\mathbb{N}\times\mathbb{N}_{\geq2}$ for 
the numbers in \eqref{eq:397467} to be  linearly independent over $\mathbb{Q}$ 
for arbitrary given infinite sets $T_{i,j}$.

\begin{thm}\label{thm:3}
Let $b\geq2$ be an integer and $\mathcal{A}$ be a subset of $\mathbb{N}\times\mathbb{N}_{\geq2}$. 
Then for any subsets $T_{i,j}$ $((i,j)\in\mathcal{A})$ of $\mathbb{N}$
and for any 
nonzero integer-valued functions $a_{i,j}(n)$ on $T_{i,j}$  
the set of the numbers 
\begin{equation}\label{eq:397467}
1, \qquad \sum_{n\in T_{i,j}}^{}\frac{a_{i,j}(n)}{b^{in^j}},\quad (i,j)\in\mathcal{A}
\end{equation}
is linearly independent over $\mathbb{Q}$  if and only if the following two conditions 
are fulfilled.   

\noindent
{\rm (i)} 
Let $(i_1,j_1), (i_2,j_2)\in\mathcal{A}$ be distinct. Then   
$i_1u^{j_1}\neq i_2v^{j_2}$ for any positive integers $u$ and $v$.   

\noindent
{\rm (ii)}There exists at most one $(i,j)\in\mathcal{A}$ with $j=2$. 
\end{thm}

Applying Theorem~\ref{thm:3}, we obtain the following Corollary~\ref{cor:34}
 which gives an another generalization of Theorem~\ref{thm:0}. 

\begin{cor}\label{cor:34}
Let $k\geq3$, $b\geq2$ and $1\leq a_1<a_2<\cdots<a_m$ be integers such that 
$\sqrt[k]{a_i/a_j}\notin\mathbb{Q}$ for any $i\neq j$. 
Let $T_i$ $(i=1,2,\dots,m)$ be any 
infinite subsets of $\mathbb{N}$. 
Then the numbers 
\[
1,\quad \sum_{n\in T_1}^{}\frac{1}{b^{a_1n^k}},\quad 
\sum_{n\in T_2}^{}\frac{1}{b^{a_2n^k}},\quad \dots, \quad 
\sum_{n\in T_m}^{}\frac{1}{b^{a_mn^k}}
\]
are linearly independent over $\mathbb{Q}$. 
\end{cor}

The present paper is organized as follows. 
Theorems~\ref{thm:1} and \ref{thm:3} will be shown in Sections~\ref{sec:2} and \ref{sec:3}, respectively. 
In the proofs, we will prove that the base-$b$ representations of nontrivial linear forms over $\mathbb{Z}$ of the numbers \eqref{eq:39746} and \eqref{eq:397467} 
contain arbitrarily long strings of zero 
without being identically zero from some point on. 
To see this, we consider the system of the simultaneous congruences 
in the proof of Theorem~\ref{thm:1}. 
This method is based on elementary arguments used in the papers 
of S.~Chowla~\cite{Ch} and P.~Erd{\H o}s~\cite{Er}. 
On the other hand, in the proof of Theorem~\ref{thm:3}, we use Mahler's result \cite{Mah} on the finiteness 
of integer solutions $x$ and $y$ for the Diophantine equation $ax^m+by^n=c$.  
These approaches are completely different from those used in the papers of V.~Kumar~\cite{Ku2,Ku1}.    
\section{Proof of Theorem~\ref{thm:1}}\label{sec:2}

We first show the following lemma.  

\begin{lem}\label{lem:1} 
Let $k\geq2$, $u\geq1$ and $v\neq0$ be integers. 
Then there are infinitely many prime numbers $p$ such that the congruence 
$uX^k+v\equiv p\pmod{p^2}$ has an integer solution. 
\end{lem}
\begin{proof}
Let $g(X):=uX^k+v\in\mathbb{Z}[X]$. 
Then there are infinitely many prime numbers $p$ 
such that $p$ divides $g(m)$ for some integer $m\geq1$; 
in other words, 
the congruence  
\begin{equation}\label{eq:47}
g(X)\equiv0\pmod{p}
\end{equation}
has an integer solution for infinitely many prime numbers $p$. 
Let $p$ be such a prime number with $p>\max\{k,u,|v|\}$ and let $x$ be an integer 
solution of \eqref{eq:47}. 
Then $g(x)\equiv0$ (mod $p$) and  
$g'(x)=kux^{k-1}\not\equiv0$ (mod $p$), since $v\neq0$. 
Hence, there exists an integer $y$ satisfying  
\begin{equation}\label{eq:347}
g'(x)y-1+g(x)/p\equiv0\pmod{p}.
\end{equation}
By \eqref{eq:347} we have   
\[
g(py+x)=u(py+x)^k+v\equiv g'(x)py+g(x)\equiv p\pmod{p^2}.
\]
This shows that  $g(X)\equiv p$ (mod $p^2$) has an integer solution $X=py+x$ and 
the proof of Lemma~\ref{lem:1} is completed. 
\end{proof}

\noindent
\textit{Proof of Theorem~\ref{thm:1}}. 
Let $m\geq2$ be an arbitrary integer. It suffices to show that the set of the numbers 
\begin{equation}\label{eq:3974611}
1, \qquad f_{i,j}(1/b)=\sum_{n\in S_{i,j}}^{}\frac{a_{i,j}(n)}{b^{in^j}}\quad (i=1,2,\dots, m,\,\,j=2,3,\dots,m)
\end{equation}
is linearly independent over $\mathbb{Q}$. 
Define 
\[
\mathcal{A}:=\{(i,j)\mid i=1,2,\dots,m,\,\,j=2,3,\dots,m\}.
\]
We fix $(i_0,j_0)\in\mathcal{A}$. 
Let $S_{i_0,j_0}$ be an infinite subset of positive integers including all prime numbers $p$ 
satisfying $p\equiv h$ (mod $d$), where 
$d:=d_{i_0,j_0}$ and $h:=h_{i_0,j_0}$ are positive integers and relatively prime. 
Let $N$ be an integer sufficiently large. 
By Lemma~\ref{lem:1} for each integer 
$\ell=1,2,\dots,N-1,N+1,\dots,2N-1$ 
there exists a prime number $p_{\ell}>N$ such that   
the congruence 
\begin{equation}\label{eq:2}
i_0X^{j_0}+\ell-N\equiv p_\ell\pmod{p_\ell^2}
\end{equation}
has an integer solution. 
Since  for each $\ell$ there are infinitely many such prime numbers, 
we may assume that the prime numbers $p_\ell$ are distinct. 
Let $x_\ell$ be a fixed integer solution of \eqref{eq:2}.  
Consider the system of the $2N-1$ simultaneous congruences 
\begin{equation}\label{eq:4}
\left\{
\begin{array}{ll}
X\equiv h&(\textrm{mod}\,\, d),\\[0.2em]
X\equiv x_\ell&(\textrm{mod}\,\, p_\ell^2),\quad \ell=1,2,\dots,N-1,N+1,\dots,2N-1.
\end{array}
\right.
\end{equation}
We find by the Chinese remainder theorem that  
there exists the least positive integer solution $x$ 
of 
\eqref{eq:4} and all integer solutions of \eqref{eq:4} are given by 
$\alpha n+x$ $(n\in \mathbb{Z})$, 
where 
\[
\alpha:=d\prod_{\substack{\ell=1\\ \ell\neq N}}^{2N-1}p_\ell^2.
\] 
Then the integers $\alpha$ and $x$ are relatively prime; indeed, if 
$p_\ell\mid x$ for some prime number $p_\ell$, then by \eqref{eq:2} and \eqref{eq:4} 
we have $p_\ell\mid \ell-N$, 
which is impossible by $1\leq |\ell-N|<N<p_\ell$. 
Moreover, by \eqref{eq:4} the integers $d$ and $x$ are relatively prime, 
since so are the integers $d$ and $h$. 
Hence, by Dirichlet's theorem on arithmetic progressions, 
there exits a positive integer $n_0$ such that $q:=\alpha n_0+x$ is a prime number 
sufficiently large. 
Note that $q\in S_{i_0,j_0}$, 
since 
$q\equiv x\equiv h\pmod{d}$. 
By \eqref{eq:2} and \eqref{eq:4} we have 
\begin{equation*}
\label{eq:21}
i_0q^{j_0}+\ell-N
\equiv p_\ell\pmod{p_\ell^2}
\end{equation*}
for each integer $\ell=1,2,\dots,N-1,N+1,\dots,2N-1$. 
This implies that the integer $i_0q^{j_0}+\ell-N$ is divisible by $p_\ell$ exactly once, 
and hence, we obtain 
\begin{equation}\label{eq:7}
i_0q^{j_0}\pm u\notin\{ik^{j}\mid k\in\mathbb{N},\,\,(i,j)\in\mathcal{A}\}
\end{equation}
for every integer $u=1,2,\dots,N-1$, since $p_\ell$ is a prime number sufficiently large.

Let  $c$ and $c_{i,j}$ $((i,j)\in\mathcal{A})$ be integers such that 
\begin{equation}\label{eq:1}
c+\sum_{(i,j)\in\mathcal{A}}c_{i,j}f_{i,j}(1/b)=0,
\end{equation}
where the sum is taken over all $i,j$ with $(i,j)\in\mathcal{A}$. 
Let   
\[
f_{i,j}(1/b)=\sum_{n\in S_{i,j}}^{}\frac{a_{i,j}(n)}{b^{in^{j}}}=\sum_{n=1}^{\infty}
\frac{e_{i,j}(n)}{b^n},
\]
where 
\begin{equation}\label{eq:8}
e_{i,j}(n):=
\left\{
\begin{array}{ll}
a_{i,j}(k)&{\rm if}\,\,n=ik^j\,\,{\rm for \,\,some \,\,integer}\,\,k\in S_{i,j},\\
0&{\rm otherwise}. 
\end{array}
\right.
\end{equation}
Then by \eqref{eq:7} and \eqref{eq:8} 
\begin{equation}\label{eq:789}
e_{i,j}(i_0q^{j_0}\pm u)=0,\quad u=1,2,\dots,N-1
\end{equation}
for any $(i,j)\in\mathcal{A}$. 
Hence, by \eqref{eq:1} and \eqref{eq:789} we obtain 
\begin{align}
P_N&:=
-c-\sum_{n=1}^{i_0q^{j_0}-N}\frac{1}{b^n}
\sum_{(i,j)\in\mathcal{A}}
{c_{i,j}e_{i,j}(n)}
\nonumber\\
&=
\frac{1}{b^{i_0q^{j_0}}}\sum_{(i,j)\in\mathcal{A}}
{c_{i,j}e_{i,j}(i_0q^{j_0})}
+
\sum_{n=i_0q^{j_0}+N}^{\infty}\frac{1}{b^n}
\sum_{(i,j)\in\mathcal{A}}
{c_{i,j}e_{i,j}(n)}.\label{eq:11}
\end{align} 
Clearly, $b^{i_0q^{j_0}-N}P_N$ is an integer. 
On the other hand, by \eqref{eq:11} 
\begin{align*}
|b^{i_0q^{j_0}-N}P_N|&\leq 
b^{-N}C_1\sum_{(i,j)\in\mathcal{A}}|c_{i,j}|+
\sum_{n=i_0q^{j_0}+N}^{\infty}\frac{C_1}{b^{n-i_0q^{j_0}+N}}
\sum_{(i,j)\in\mathcal{A}}{|c_{i,j}|}\nonumber\\
&\leq C_2b^{-N}+2C_2b^{-2N}\to 0\quad (N\to \infty),
\end{align*}
where $C_1:=\max_{i,j}\{|a_{i,j}(n)|\mid n\in S_{i,j}\}$ and $C_2:=C_1\sum_{(i,j)\in\mathcal{A}}|c_{i,j}|$. 
Thus, we obtain $P_N=0$ for every large integer $N$. 
Therefore, by \eqref{eq:11} we find that 
\[
Q_N:=\sum_{(i,j)\in\mathcal{A}}c_{i,j}e_{i,j}(i_0q^{j_0})=
-
\sum_{n=i_0q^{j_0}+N}^{\infty}\frac{1}{b^{n-i_0q^{j_0}}}
\sum_{(i,j)\in\mathcal{A}}
{c_{i,j}e_{i,j}(n)}
\]
is an integer for every large $N$ and 
\[
|Q_N|\leq 2C_2b^{-N}\to 0\quad (N\to \infty). 
\]
This implies that 
\begin{equation}\label{eq:31}
Q_N=\sum_{(i,j)\in\mathcal{A}}c_{i,j}e_{i,j}(i_0q^{j_0})=0
\end{equation}
for every large integer $N$. 
Let $(i,j)\in\mathcal{A}$ be satisfy $i_0q^{j_0}=ik^{j}$ 
for some integer $k\in S_{i,j}$. 
Then we obtain  $i\mid i_0$ 
and 
$j\mid j_0$, since $q$ is a prime number sufficiently large.   
Hence, the equality \eqref{eq:31} is  written as  
\begin{equation}\label{eq:60}
\sum_{
\substack
{(i,j)\in\mathcal{A}\\i\leq i_0,j\leq j_0}
}
c_{i,j}e_{i,j}(i_0q^{j_0})=0.
\end{equation}
Now we show $c_{i,j}=0$ for any $(i,j)\in\mathcal{A}$ by induction on $i+j(\geq3)$. 
We first apply the above arguments to the case where $(i_0,j_0)=(1,2)\in\mathcal{A}$ and 
the set $S_{1,2}$. 
Then  
there exists a prime number $q:=q_{1,2}\in S_{1,2}$ 
satisfying \eqref{eq:60}, so that, 
$c_{1,2}e_{1,2}(q^2)=0$. 
Since $e_{1,2}(q^2)=a_{1,2}(q)\neq0$, we obtain $c_{1,2}=0$. 
Let $M\geq3$ be an integer and suppose that $c_{i,j}=0$ for any $(i,j)\in\mathcal{A}$ with $i+j\leq M$. 
Then, applying the above arguments again to the case where $(i_0,j_0)\in\mathcal{A}$ with $i_0+j_0=M+1$ 
and the set $S_{i_0,j_0}$,  
we obtain \eqref{eq:60} for some prime number $q:=q_{i_0,j_0}\in S_{i_0,j_0}$. 
Hence, by the induction hypothesis 
\[
0=\sum_{
\substack
{(i,j)\in\mathcal{A}\\i\leq i_0,j\leq j_0}
}
c_{i,j}e_{i,j}(i_0q^{j_0})
=c_{i_0,j_0}e_{i_0,j_0}(i_0q^{j_0}).
\]
Since $e_{i_0,j_0}(i_0q^{j_0})=a_{i_0,j_0}(q)\neq0$, we have $c_{i_0,j_0}=0$. 
Thus, we obtain $c_{i,j}~=0$ for every $(i,j)\in\mathcal{A}$, and so $c=0$ by \eqref{eq:1}.  
The proof of Theorem~\ref{thm:1} is completed. \qed 

\section{Proof of Theorem~\ref{thm:3}}\label{sec:3}
We first show that if both conditions~(i) and (ii) do not hold, 
then  the numbers in~\eqref{eq:397467} 
can be {linearly dependent} over $\mathbb{Q}$ for 
suitable infinite sets $T_{i,j}$ and the constant functions $a_{i,j}(n)\equiv1$ $(n\in T_{i,j})$. 
If condition~(i) does not hold, then there exit distinct pairs $(i_1,j_1),(i_2,j_2)\in\mathcal{A}$ and positive integers $u,v$ such that $i_1u^{j_1}=i_2v^{j_2}$. 
Hence, putting $T_1:=\{u2^{j_2m}\mid m\in\mathbb{N}\}$ and $T_2:=\{v2^{j_1m}\mid m\in\mathbb{N}\}$, 
we obtain the equality 
$\sum_{n\in T_1}^{}{b^{-i_1n^{j_1}}}=\sum_{n\in T_2}^{}{b^{-i_2n^{j_2}}}.
$
Next we assume that condition~(i) holds while condition~(ii) does not. 
Then there exist $(i_1,2),(i_2,2)\in\mathcal{A}$ such that $i_1i_2$ is not square. 
Consider the set of positive integer solutions of the Pell equation 
$x^2-i_1i_2y^2=1$:
\[
\mathcal{S}:=\left\{
(x,y)\in\mathbb{N}^2\mid x^2-i_1i_2y^2=1\right\}. 
\]
As is well known, $\mathcal{S}$ is an infinite set. 
Hence, noting that $i_1x^2-i_2(i_1y)^2=i_1$ for $(x,y)\in\mathcal{S}$, we have  
\[
\sum_{n\in T_1}^{}\frac{1}{b^{i_1n^2}}=\sum_{n\in T_2}^{}\frac{1}{b^{i_2n^2+i_1}}=
\frac{1}{b^{i_1}}\sum_{n\in T_2}^{}\frac{1}{b^{i_2n^2}},
\]
where $T_1:=\{x\mid (x,y)\in\mathcal{S}\}$ and 
$T_2:=\{i_1y\mid (x,y)\in\mathcal{S}\}$.  
Thus, our assertion is proved.

In what follows, we assume that both conditions (i) and (ii) hold. 
We fix $(i_0,j_0)\in \mathcal{A}$ and let $N\geq1$ be an integer sufficiently large. 
Consider the Diophantine equations
\begin{equation}\label{eq:0874}
i_0x^{j_0}-iy^j=\pm u,
\end{equation}
where $(i,j)\in\mathcal{A}$ and $u=1,2,\dots,N-1$. 
If $j_0=j=~2$, we have $i_0=i$  by condition~(ii). Then clearly 
the equation \eqref{eq:0874} has only finitely many integer solutions $x$ and $y$. 
Let $\max\{j_0,j\}\geq 3$. Then K.~Mahler~\cite{Mah} proved that  
the greatest prime factor of $i_0x^{j_0}-iy^j$ tends to infinity 
as $\max \{|x|,|y|\}\to\infty$ with relatively prime integers $x$ and $y$. 
Hence, also in this case, the equation 
\eqref{eq:0874} has only finitely many integer solutions $x$ and $y$. 
Thus, there exists a positive integer $M$ such that 
\begin{equation}\label{eq:8462}
i_0 t^{j_0}\pm u\notin\{
ik^{j}\mid k\in\mathbb{N}, (i,j)\in\mathcal{A}
\}
\end{equation}
holds for every integer $t>M$ and 
for all integers $u=1,2,\dots,N-1$. 

Let $c$ and $c_{i,j}$ $((i,j)\in\mathcal{A})$ be integers such that 
\begin{equation}\label{eq:1123}
c+\sum_{(i,j)\in\mathcal{A}}^{}c_{i,j}\delta_{i,j}=0,
\end{equation}
where  
\[
\delta_{i,j}:=\sum_{n\in T_{i,j}}^{}
\frac{a_{i,j}(n)}{b^{in^{j}}}=\sum_{n=1}^{\infty}
\frac{e_{i,j}(n)}{b^n}
\]
and 
\begin{equation}\label{eq:8123}
e_{i,j}(n):=
\left\{
\begin{array}{ll}
a_{i,j}(k)&{\rm if}\,\,n=ik^j\,\,{\rm for \,\,some \,\,integer}\,\,k\in T_{i,j},\\
0&{\rm otherwise}. 
\end{array}
\right.
\end{equation}
Let $t\in T_{i_0,j_0}$ $(t>M)$ be an integer satisfying \eqref{eq:8462}. 
Then 
by \eqref{eq:8462} and \eqref{eq:8123} 
\begin{equation*}\label{eq:78945}
e_{i,j}(i_0t^{j_0}\pm u)=0,\quad u=1,2,\dots,N-1
\end{equation*}
for any $(i,j)\in\mathcal{A}$. 
Hence, similarly as in the proof of Theorem~\ref{thm:1}, we can derive  
\begin{equation}\label{eq:745}
\sum_{(i,j)\in\mathcal{A}}c_{i,j}e_{i,j}(i_0t^{j_0})=0
\end{equation}
for every large integer $N$. 
Clearly, we have $e_{i_0,j_0}(i_0t^{j_0})=a_{i_0,j_0}(t)\neq0$. 
Moreover, if  $(i,j)\neq(i_0,j_0)$, then by condition~(i) we have 
$i_0t^{j_0}\neq iv^j$ for every positive integer $v$, 
and so $e_{i,j}(i_0t^{j_0})=0$. 
Therefore, by \eqref{eq:745} we obtain $c_{i_0,j_0}=0$. 
Since the integer pair $(i_0,j_0)$ is chosen arbitrarily, 
we obtain $c_{i,j}=0$ for every $(i,j)\in\mathcal{A}$, and hence, $c=0$ by \eqref{eq:1123}.  
Theorem~\ref{thm:3} is proved. \qed  

\subsection*{Acknowledgements}
The authors would like to express their sincere gratitude to Professor Hajime Kaneko for 
valuable comments and suggestions on this manuscript. 
The authors are also grateful to Professor Takafumi Miyazaki for 
pointing out the paper of K.~Mahler~\cite{Mah}. 
This work was partly supported 
by JSPS KAKENHI Grant Number   
JP18K03201. 

\end{document}